\documentclass[11pt]{article}

\usepackage{amsmath,amssymb,amsthm} 
\usepackage[mathscr]{eucal} 
\usepackage{latexsym}
\theoremstyle{plain}
\usepackage{enumerate}
\usepackage{hyperref}
\usepackage{color}
\usepackage{authblk}

\usepackage[T1]{fontenc}
\usepackage[bitstream-charter]{mathdesign}
\usepackage{eucal}

\newtheorem{thm}{Theorem}
\newtheorem{lem}[thm]{Lemma}
\newtheorem{rem}{Remark}

\newcommand{\norm}[1]{{\left\Vert #1 \right\Vert}}
\newcommand{\set}[1]{{\left\lbrace #1 \right\rbrace}}

\DeclareMathOperator*{\divg}{div}
\DeclareMathOperator*{\supp}{supp}

\title{New Liouville-type theorem for the stationary tropical climate model}

\author{Youseung Cho}

\author{Hyunjin In}

\author{Minsuk Yang}

\affil{Yonsei University, Department of Mathematics \\
50 Yonseiro, Seodaemungu, Seoul 03722, Republic of Korea\\
E-mail: m.yang@yonsei.ac.kr}

\begin{document} 

\date{\today}

\maketitle

\begin{abstract} 
We study the Liouville-type theorem for smooth solutions to the steady 3D tropical climate model. 
We prove the Liouville-type theorem if a smooth solution satisfies a certain growth condition in terms of $L^p$-norm on annuli, which improves the previous results, Theorem 1.1 (Math. Methods Appl. Sci. 44, 2021) by Ding and Wu, and Theorem 1.1 and Theorem 1.2 (Appl. Math. Lett. 138, 2023) by Yuan and Wang.

\end{abstract}

\thanks{ 
\endgraf
}

\maketitle

\section{Introduction}
\label{S1}

The nonlinear partial differential equations
\begin{equation}
\label{E11}
\begin{split}
- \Delta u + (u \cdot \nabla ) u + \nabla \pi 
+ \divg (v \otimes v) &= 0, \\
- \Delta v + (u \cdot \nabla ) v + \nabla \theta 
+ (v \cdot \nabla )u &= 0, \\
- \Delta \theta + u \cdot \nabla \theta 
+ \divg v &= 0, \\
\divg u &= 0,
\end{split}
\end{equation}
in $\mathbb{R}^3$, describe the stationary tropical climate model.
Here, $u = (u_1,u_2,u_3)$ is the barotropic mode, $v = (v_1, v_2, v_3)$ is the first baroclinic mode of vector velocity, $\theta$ is the temperature, and $\pi$ is the pressure.

One of the most famous Liouville-type theorems is that if $f$ solves the Laplace equation on $\mathbb{R}^3$ and $f \in L^\infty(\mathbb{R}^3)$, then $f$ must be constant.
In general, Liouville-type theorems are about finding some conditions to show that solutions to some PDEs become trivial.
Recently, there have been many efforts to establish Liouville-type theorems for various fluid equations.
For the stationary Navier–Stokes equations, one can find interesting results, for example, in 
Chae \cite{MR3162482}, Seregin \cite{MR3538409}, Kozono--Terasawa--Wakasugi \cite{MR3571910}, Chae--Wolf \cite{MR3959933}, Tsai \cite{MR4354995}, and Cho--Choi--Yang \cite{MR4572397}.
For the stationary tropical climate model, there are only a few results, \cite{MR4342846}, \cite{MR4522953}, and \cite{MR4530222}.
Authors in \cite{MR4342846} proved Liouville-type theorems for \eqref{E11} if a smooth solution belongs to certain Lebesgue spaces or Lorentz spaces. 
Later on, they proved Liouville-type theorems assuming that a smooth solution belongs to the mixed local Morrey spaces.
Recently, authors in \cite{MR4522953} extended some results in \cite{MR4342846} in terms of the local Morrey spaces.

Our aim is to establish an improved Liouville-type theorem for the stationary tropical climate model.
We will use the following notation
\begin{equation}
\label{E12}
G(f;\alpha,\rho) : = \rho^{- \left(\frac{2}{\alpha} - \frac{1}{3} \right)} \norm{f}_{L^\alpha(B(2\rho) \setminus B(\rho/2))}
\quad \text{and} \quad
H(f;\alpha,\rho) : = \rho^{- \left( \frac{3}{2\alpha} - \frac{1}{4} \right)} \norm{f}_{L^\alpha(B(2\rho) \setminus B(\rho/2))}, 
\end{equation}
where for functions $f \in L^1_{loc}(\mathbb{R}^3)$, integrability exponents $1 \le \alpha < \infty$, and radii $1 < \rho < \infty$.

Here is our main result.

\begin{thm}
\label{T1}
Let $(u,v,\theta)$ be a smooth solution to \eqref{E11}.
Then $u = v = \theta = 0$ if one of the following conditions holds.
\begin{enumerate}
\item
For some $3/2 < \alpha < 3$ and $1 \le \beta, \gamma < 2$,
\begin{equation}
\label{A1}
\liminf_{\rho \to \infty} \left(G(u;\alpha, \rho) + H(v;\beta, \rho) + H(\theta;\gamma, \rho) \right) < \infty.
\end{equation}
\item
\begin{equation}
\label{A2}
\liminf_{\rho \to \infty} \left(G(u;3, \rho) + H(v;2, \rho) \right) = 0 \quad \text{and} \quad 
\limsup_{\rho \to \infty} H(\theta;2, \rho) < \infty.
\end{equation}
\item
\begin{equation}
\label{A3}
\liminf_{\rho \to \infty} \left(G(u;3, \rho) + H(\theta;2, \rho) \right) 
= 0 
\quad \text{and} \quad 
\limsup_{\rho \to \infty} H(v;2, \rho) < \infty.
\end{equation}
\end{enumerate}
\end{thm}

\begin{rem}
\label{r1}
\begin{enumerate}
\item
In \cite{MR4342846}, the conditions of Theorem 1.1 are $u \in L^p (\mathbb{R}^3)$ and $v,\theta \in L^p (\mathbb{R}^3) \cap L^3 (\mathbb{R}^3)$ with $3 \le p \le 9/2$. 
Using the Jensen inequality, one can easily check that \eqref{A2} and \eqref{A3} are weaker than $u \in L^p(\mathbb{R}^3)$ and $v, \theta \in L^q(\mathbb{R}^3)$ for some $3 \le p \le 9/2$ and $2 \le q \le 3$. 
Indeed, we removed the additional conditions $v, \theta \in L^p(\mathbb{R}^3)$.

\item
For $q \ge 0$ and $1 < p < \infty$, the local Morrey space $M_q^p (\mathbb{R}^3)$ is defined to be the set of functions satisfying $\norm{f}_{M_q^p} := \sup_{R \ge 1} R^{-\frac{q}{p}} \norm{f}_{L^p(B(R))} < \infty$ and $M_{q, 0}^p (\mathbb{R}^3)$ is a subspace of $M_q^p (\mathbb{R}^3)$ satisfying 
\[\lim_{R \to \infty} R^{-\frac{q}{p}} \norm{f}_{L^p(B(R) \setminus B(R/2))} = 0.\]
In \cite{MR4522953}, the conditions of Theorem 1.1 are $(u,v,\theta) \in M_{q,0}^p (\mathbb{R}^3)$ and $v \in M_{q,0}^2 (\mathbb{R}^3)$, $\theta \in M_q^2 (\mathbb{R}^3)$ or $v \in M_q^2 (\mathbb{R}^3)$, $\theta \in M_{q,0}^2 (\mathbb{R}^3)$ with $3 \le p < 9/2$, $0 < q \le 1$, and $3 q + 2 p \le 9$, and the conditions of Theorem 1.2 are $(u,v,\theta) \in M_q^p (\mathbb{R}^3)$ and $(v,\theta) \in M_q^2 (\mathbb{R}^3)$ with $3 \le p < 9/2$, $0 < q < 1$, and $3 q + 2 p < 9$.

\item
Using the Jensen inequality, one can easily check that \eqref{A2} is weaker than the conditions $u \in M^{p_1}_{q_1,0}(\mathbb{R}^3)$, $v \in M^{p_2}_{q_2,0}(\mathbb{R}^3)$, and $\theta \in M^{p_2}_{q_2}(\mathbb{R}^3)$ with $3 \le p_1 \le 9/2$, $3q_1 \le 9 - 2p_1$, $2 \le p_2 \le 3$, and $q_2 \le 3 - p_2$.
Similarly, \eqref{A3} is weaker than the conditions $u \in M^{p_1}_{q_1,0}(\mathbb{R}^3)$, $v \in M^{p_2}_{q_2}(\mathbb{R}^3)$, and $\theta \in M^{p_2}_{q_2,0}(\mathbb{R}^3)$ with the same indices, and 
\eqref{A1} is weaker than the conditions $u \in M^{p_1}_{q_1}(\mathbb{R}^3)$ with $3 \le p_1 < 9/2$ and $3q_1 < 9 - 2p_1$ and $v, \theta \in M^{p_2}_{q_2}(\mathbb{R}^3)$ with $2 \le p_2 < 3$ and $q_2 < 3-p_2$.
Therefore, Theorem 1.1 and Theorem 1.2 of \cite{MR4522953} are now special cases of Theorem \ref{T1} in this paper.
\end{enumerate}
\end{rem}

\section{Preliminary}
\label{S2}

Throughout this paper, we shall use the following notations.

\begin{itemize}
\item
For $1 \le p \le \infty$, we denote the H\"older conjugate of p by $p' := p/(p-1)$.
\item
For $0 < R < \infty$, we denote open balls and annuli by
\[
B(R) = \set{x \in \mathbb{R}^3 : 0 \le |x| < R}
\quad \text{and} \quad 
A(R) = \set{x \in \mathbb{R}^3 : R/2 < |x| < R}.
\]
\item
We denote the Lebesgue measure of a measurable set $\Omega \subset \mathbb{R}^3$ by $|\Omega|$ and the Lebesgue integral of $f$ over $\Omega$ by $\int_\Omega f dx$.

\item
We denote $L_0^p(\Omega) = \set{f \in L^p (\Omega) : f_\Omega = 0}$, where the average value of $f$ over $\Omega$ is given by $f_\Omega = \frac{1}{|\Omega|} \int_\Omega f dx$.
\item
We denote $A \lesssim_t B$ if $|A| \le C(t) |B|$ for some positive constant $C(t)$ depending only on some parameter $t$.
\item
Let $\varphi_{r,R} \in C_c^{\infty} (B(R))$ be a radially decreasing function satisfying $\varphi_{r,R} = 1$ on $B(r)$ and 
\[
(R-r) |\nabla \varphi_{r,R}| + (R-r)^2 |\nabla^2 \varphi_{r,R}| < \infty.
\]
\end{itemize}

We recall the Poincar\'e--Sobolev inequality.

\begin{lem}[Theorem 3.15 of \cite{MR1962933}]
\label{L11}
Let $\Omega \subset \mathbb{R}^n$ be a bounded connected open set with Lipschitz-continuous boundary $\partial \Omega$.
Then there exists a positive constant $C(n,p,\Omega)$ such that if $p < n$, then we have for every $f \in W^{1,p} (\Omega)$,
\begin{equation}
\label{E15}
\norm{f - f_{\Omega}}_{L^\frac{np}{n-p} (\Omega)}
\le C(n,p,\Omega) \norm{\nabla f}_{L^p (\Omega)}.
\end{equation}
\end{lem}

\begin{rem}
\label{r2}
Using the scaling argument, the constant in \eqref{E15} does not depend on $R$ when $\Omega = A(R) \subset \mathbb{R}^3$.
Using \eqref{E15} and the Jensen inequality, we obtain that for $1 \le p \le \infty$,
\begin{equation}
\label{PS}
\norm{f}_{L^6(A(R))}
\lesssim \norm{\nabla f}_{L^2(A(R))}
+ R^{\frac{1}{2} - \frac{3}{p}} \norm{f}_{L^p(A(R))},
\end{equation}
where the implied constant is now absolute.
\end{rem}

We end this section by presenting the following form of local energy inequality.

\begin{lem}
\label{L12}
Let $(u,v,\theta)$ be a smooth solution to \eqref{E11}.
Let $1 < \alpha < \infty$ and $E(\Omega) := \norm{\nabla u}_{L^2 (\Omega)}^2 + \norm{\nabla v}_{L^2 (\Omega)}^2 + \norm{\nabla \theta}_{L^2 (\Omega)}^2.$
Then for $1 \le \rho \le r < R \le 2\rho$, we have
\begin{align}
\begin{split}
\label{LEI}
E(B(r))
&\le C(\alpha) (R-r)^{-1} \left( (R-r)^{-1} \rho^\frac{3}{\alpha} + \norm{u}_{L^{\alpha} (A(R))} \right)
\left( \norm{u}_{L^{2\alpha'} (A(R))}^2 
+ \norm{v}_{L^{2\alpha'} (A(R))}^2 
+ \norm{\theta}_{L^{2\alpha'} (A(R))}^2 \right) \\
&\quad + C (R-r)^{-1} \norm{v}_{L^2(A(R))} \norm{\theta}_{L^2 (A(R))} 
+ \frac{1}{8} E(A(R)).
\end{split}
\end{align}
\end{lem}

\begin{proof}
Since $u$ is smooth with $\divg u = 0$ and $\supp \nabla \varphi_{r,R} \subset A(R)$, we have $u \cdot \nabla \varphi_{r,R} \in L_0^\alpha(A(R))$.
Let $w = \mathcal{B} (u \cdot \nabla \varphi_{r,R})$ in $A(R)$, where $\mathcal{B}$ denotes the Bogovskii operator (see \cite[Lemma 3]{MR4354995} or \cite[Lemma 4]{MR4572397}).
Then $\divg w = u \cdot \nabla \varphi_{r,R}$ and for $1 < p < \infty$,
\begin{equation}
\label{Bog}
\norm{\nabla w}_{L^p (A(R))} \lesssim_p (R-r)^{-1} \norm{u}_{L^p (A(R))}.
\end{equation}
Multiplying the first equation of \eqref{E11} by $(u \varphi_{r,R} - w)$, the second equation of \eqref{E11} by $v \varphi_{r,R}$, and the third equation of \eqref{E11} by $\theta \varphi_{r,R}$, integrating by parts with $\divg u = 0$, and then summing the resulting equations yields 
\begin{align*}
&\int \left(|\nabla u|^2 + |\nabla v|^2 + |\nabla \theta|^2\right) \varphi_{r,R} dx \\
&= \frac{1}{2} \int \left(|u|^2 + |v|^2 + |\theta|^2\right) \Delta \varphi_{r,R} dx
+ \frac{1}{2} \int \left(|u|^2 + |v|^2 + |\theta|^2\right) u \cdot \nabla \varphi_{r,R} dx
+ \int (v \cdot u) (v \cdot \nabla \varphi_{r,R}) dx \\
&\quad + \int \theta v \cdot \nabla \varphi_{r,R} dx
- \int_{A(R)} (u \otimes u + v \otimes v) : \nabla w dx
+ \int_{A(R)} \nabla u : \nabla w dx.
\end{align*}
We use $\supp (\nabla \varphi_{r,R}) \subset A(R)$, the H\"older inequality, and $R \le \rho$ to get 
\begin{align*}
&\int \left(|u|^2 + |v|^2 + |\theta|^2\right) \Delta \varphi_{r,R} dx
+ \int \left(|u|^2 + |v|^2 + |\theta|^2\right) u \cdot \nabla \varphi_{r,R} dx
+ \int (v \cdot u) (v \cdot \nabla \varphi_{r,R}) dx
+ \int \theta v \cdot \nabla \varphi_{r,R} dx \\
&\lesssim (R-r)^{-1} \left( (R-r)^{-1} \rho^\frac{3}{\alpha} + \norm{u}_{L^{\alpha} (A(R))} \right)
\left( \norm{u}_{L^{2\alpha'} (A(R))}^2 
+ \norm{v}_{L^{2\alpha'} (A(R))}^2 
+ \norm{\theta}_{L^{2\alpha'} (A(R))}^2 \right) \\
&\quad + (R-r)^{-1} \norm{v}_{L^2(A(R))} \norm{\theta}_{L^2 (A(R))}.
\end{align*}
By the H\"older inequality and \eqref{Bog}, we get 
\begin{align*}
\int_{A(R)} (u \otimes u + v \otimes v) : \nabla w dx
&\le \left(\norm{u}_{L^{2\alpha'} (A(R))}^2 + \norm{v}_{L^{2\alpha'} (A(R))}^2\right) \norm{\nabla w}_{L^\alpha (A(R))} \\
&\lesssim_\alpha (R-r)^{-1} \norm{u}_{L^\alpha (A(R))} \left(\norm{u}_{L^{2\alpha'} (A(R))}^2 
+ \norm{v}_{L^{2\alpha'} (A(R))}^2\right).
\end{align*}
By \eqref{Bog}, the H\"older inequality, and the Young inequality, we obtain
\begin{align*}
\int_{A(R)} \nabla u : \nabla w dx
&\lesssim_\alpha (R-r)^{-1} \rho^{\frac{3}{2\alpha}} 
\norm{u}_{L^{2\alpha'} (A(R))} 
\norm{\nabla u}_{L^2 (A(R))} \\
&\le \frac{1}{8} \norm{\nabla u}_{L^2 (A(R))}^2 + C(\alpha) (R-r)^{-2} \rho^{\frac{3}{\alpha}} \norm{u}_{L^{2\alpha'} (A(R))}^2.
\end{align*}
Combining all the estimates, we get \eqref{LEI}.
\end{proof}

\section{Proof of Theorem \ref{T1}}
\label{S3}

We first assume \eqref{A1}.
Then we have an increasing sequence $\set{\rho_j}$ such that $1 < \rho_1$, and for all $j$,
\begin{equation}
\label{E31}
G(u;\alpha,\rho_j) + H(v;\beta,\rho_j) 
+ H(\theta;\gamma,\rho_j) \le M < \infty.
\end{equation}
Fix $r < R$ satisfying $\rho_j \le r < R \le 2\rho_j$.
We shall estimate the right side of \eqref{LEI}.
As $A(R) \subset B(2\rho_j) \setminus B(\rho_j/2)$, by \eqref{E31}, we have 
\begin{equation}
\label{E32}
(R-r)^{-1} \rho_j^\frac{3}{\alpha} + \norm{u}_{L^{\alpha} (A(R))}
\lesssim (R-r)^{-1} \rho_j \left( \rho_j^{\frac{3}{\alpha} - 1} + \rho_j^{\frac{2}{\alpha} - \frac{1}{3}} M \right)
\lesssim_M (R-r)^{-1} \rho_j^{\frac{2}{\alpha} + \frac{2}{3}}.
\end{equation}
Here, we use the assumption $\alpha > 3/2$ in the last inequality.
By the H\"older inequality, \eqref{PS}, and \eqref{E31}, we get
\begin{equation}
\label{E33}
\norm{u}_{L^{2\alpha'} (A(R))}^2 
\lesssim \norm{u}_{L^{\alpha} (A(R))}^{\frac{2(2\alpha-3)}{6-\alpha}} \norm{u}_{L^{6} (A(R))}^{\frac{6(3-\alpha)}{6-\alpha}}
\lesssim_M \rho_j^{\frac{2(2\alpha-3)}{3\alpha}} 
\left( \norm{\nabla u}_{L^2(A(R))} + \rho_j^{-\frac{1}{\alpha} + \frac{1}{6}} \right)^{\frac{6(3-\alpha)}{6-\alpha}}.
\end{equation}
Combining \eqref{E32} and \eqref{E33}, we have
\begin{equation}
\label{E34}
(R-r)^{-1} \left( (R-r)^{-1} \rho^\frac{3}{\alpha} + \norm{u}_{L^{\alpha} (A(R))} \right) \norm{u}_{L^{2\alpha'} (A(R))}^2
\lesssim_M (R-r)^{-2} \rho_j^{2} 
\left( \norm{\nabla u}_{L^2(A(R))} + \rho_j^{-\frac{1}{\alpha} + \frac{1}{6}} \right)^{\frac{6(3-\alpha)}{6-\alpha}}.
\end{equation}
Using the Young inequality, we get 
\begin{equation}
\label{E35}
(R-r)^{-1} \left( (R-r)^{-1} \rho^\frac{3}{\alpha} + \norm{u}_{L^{\alpha} (A(R))} \right) \norm{u}_{L^{2\alpha'} (A(R))}^2
\le \frac{1}{8} E(A(R)) + C(M) \left( (R-r)^{-1} \rho_j \right)^{\frac{2(6-\alpha)}{2\alpha-3}}.
\end{equation}
Similar to \eqref{E33}, we use the H\"older inequality, \eqref{PS}, and \eqref{E31} to get
\begin{equation}
\label{E36}
\norm{v}_{L^{2\alpha'} (A(R))}^2 
\lesssim \norm{v}_{L^{\beta} (A(R))}^{\frac{2\beta(2\alpha-3)}{\alpha(6-\beta)}} \norm{v}_{L^{6} (A(R))}^{\frac{6(2\alpha + \beta - \alpha\beta)}{\alpha(6-\beta)}}
\lesssim_M \rho_j^{\frac{2\alpha-3}{2\alpha}}
\left( \norm{\nabla v}_{L^2(A(R))} + \rho_j^{-\frac{3}{2\beta} + \frac{1}{4}} \right)^{\frac{6(2\alpha + \beta - \alpha\beta)}{\alpha(6-\beta)}}.
\end{equation}
and
\begin{equation}
\label{E37}
\norm{\theta}_{L^{2\alpha'} (A(R))}^2 
\lesssim \norm{\theta}_{L^{\gamma} (A(R))}^{\frac{2\gamma(2\alpha-3)}{\alpha(6-\gamma)}} \norm{\theta}_{L^{6} (A(R))}^{\frac{6(2\alpha + \gamma - \alpha\gamma)}{\alpha(6-\gamma)}}
\lesssim_M \rho_j^{\frac{2\alpha-3}{2\alpha}} 
\left( \norm{\nabla \theta}_{L^2(A(R))} + \rho_j^{-\frac{3}{2\gamma} + \frac{1}{4}} \right)^{\frac{6(2\alpha + \gamma - \alpha\gamma)}{\alpha(6-\gamma)}}.
\end{equation}
Combining \eqref{E32}, \eqref{E36}, and \eqref{E37}, we have
\begin{align}
\begin{split}
\label{E38}
&(R-r)^{-1} \left( (R-r)^{-1} \rho^\frac{3}{\alpha} + \norm{u}_{L^{\alpha} (A(R))} \right)
\left( \norm{v}_{L^{2\alpha'} (A(R))}^2 
+ \norm{\theta}_{L^{2\alpha'} (A(R))}^2 \right) \\
&\lesssim_M (R-r)^{-2} \rho_j^{\frac{5}{3} + \frac{1}{2\alpha}} 
\left( \left( \norm{\nabla v}_{L^2(A(R))} + \rho_j^{-\frac{3}{2\beta} + \frac{1}{4}} \right)^{\frac{6(2\alpha + \beta - \alpha\beta)}{\alpha(6-\beta)}} 
+ \left( \norm{\nabla \theta}_{L^2(A(R))} + \rho_j^{-\frac{3}{2\gamma} + \frac{1}{4}} \right)^{\frac{6(2\alpha + \gamma - \alpha\gamma)}{\alpha(6-\gamma)}} \right).
\end{split}
\end{align}
After some routine calculations involving the Young inequality, we can easily obtain
\begin{align}
\begin{split}
\label{E39}
&(R-r)^{-1} \left( (R-r)^{-1} \rho^\frac{3}{\alpha} + \norm{u}_{L^{\alpha} (A(R))} \right)
\left( \norm{v}_{L^{2\alpha'} (A(R))}^2 
+ \norm{\theta}_{L^{2\alpha'} (A(R))}^2 \right) \\
&\le \frac{1}{8} E(A(R))
+ C(M) \left( (R-r)^{-1} \rho_j \right)^{\frac{2\alpha(6-\beta)}{\beta(2\alpha-3)}}
+ C(M) \left( (R-r)^{-1} \rho_j \right)^{\frac{2\alpha(6-\gamma)}{\gamma(2\alpha-3)}}.
\end{split}
\end{align}
By the H\"older inequality, \eqref{E31} and \eqref{PS}, we have
\[
\norm{v}_{L^{2}(A(R))} \lesssim \norm{v}_{L^{\beta}(A(R))}^{\frac{2\beta}{6-\beta}} \norm{v}_{L^{6}(A(R))}^{\frac{3(2-\beta)}{6-\beta}}
\lesssim_M \rho_j^{\frac{1}{2}} \left( \norm{\nabla v}_{L^2(A(R))} + \rho_j^{-\frac{3}{2\beta} + \frac{1}{4}} \right)^{\frac{3(2-\beta)}{6-\beta}}
\]
and
\[
\norm{\theta}_{L^{2}(A(R))} \lesssim \norm{\theta}_{L^{\gamma}(A(R))}^{\frac{2\gamma}{6-\gamma}} \norm{\theta}_{L^{6}(A(R))}^{\frac{3(2-\gamma)}{6-\gamma}}
\lesssim_M \rho_j^{\frac{1}{2}} \left( \norm{\nabla \theta}_{L^2(A(R))} + \rho_j^{-\frac{3}{2\gamma} + \frac{1}{4}} \right)^{\frac{3(2-\gamma)}{6-\gamma}}.
\]
Hence
\begin{equation}
\label{E310}
(R-r)^{-1} \norm{v}_{L^2(A(R))} \norm{\theta}_{L^2 (A(R))}
\lesssim_M (R-r)^{-1} \rho_j 
\left( \norm{\nabla v}_{L^2(A(R))} + \rho_j^{-\frac{3}{2\beta} + \frac{1}{4}} \right)^{\frac{6-3\beta}{6-\beta}} 
\left( \norm{\nabla \theta}_{L^2(A(R))} + \rho_j^{-\frac{3}{2\gamma} + \frac{1}{4}} \right)^{\frac{6-3\gamma}{6-\gamma}}.
\end{equation}
Using the Young inequality, we can obtain
\begin{equation}
\label{E311}
(R-r)^{-1} \norm{v}_{L^2(A(R))} \norm{\theta}_{L^2 (A(R))}
\le \frac{1}{8} E(A(R))
+ C(M) \left( (R-r)^{-1} \rho_j \right)^{\frac{(6-\beta)(6-\gamma)}{\beta(6-\gamma) + \gamma(6-\beta)}}.
\end{equation}
Combining \eqref{LEI}, \eqref{E35}, \eqref{E39}, and \eqref{E311}, we can deduce that
\[
E(B(r)) \le \frac{1}{2} E(B(R)) + C(M) \left( (R-r)^{-1} \rho_j \right)^k
\]
holds for all $r < R$ in the range $[\rho_j, 2\rho_j]$ and $k = \max \set{\frac{2(6-\alpha)}{2\alpha-3}, \frac{2\alpha(6-\beta)}{\beta(2\alpha-3)}, \frac{2\alpha(6-\gamma)}{\gamma(2\alpha-3)}, \frac{(6-\beta)(6-\gamma)}{\beta(6-\gamma) + \gamma(6-\beta)}}$.
Applying the standard iteration lemma (see \cite[Lemma 2]{MR4591749} or \cite[V. Lemma 3.1]{MR717034}), we have
\[
E(B(\rho_j)) < \infty
\]
for all $j$ and hence $\nabla u, \nabla v, \nabla \theta \in L^2(\mathbb{R}^3)$.
To conclude the proof, we use \eqref{LEI} with $r= \rho_j$ and $R = 2\rho_j$:
\begin{align}
\begin{split}
\label{LEI2}
E(B(\rho_j))
&\lesssim \rho_j^{-1} \left( \rho_j^{-1 + \frac{3}{\alpha}} + \norm{u}_{L^{\alpha} (A(2\rho_j))} \right)
\left( \norm{u}_{L^{2\alpha'} (A(2\rho_j))}^2 
+ \norm{v}_{L^{2\alpha'} (A(2\rho_j))}^2 
+ \norm{\theta}_{L^{2\alpha'} (A(2\rho_j))}^2 \right) \\
&\quad + \rho_j^{-1} \norm{v}_{L^2(A(2\rho_j))} \norm{\theta}_{L^2 (A(2\rho_j))} 
+ E(A(2\rho_j)).
\end{split}
\end{align}
Employing \eqref{E34} and \eqref{E38}, we have
\begin{align*}
&\rho_j^{-1} \left( \rho_j^{-1 + \frac{3}{\alpha}} + \norm{u}_{L^{\alpha} (A(2\rho_j))} \right)
\left( \norm{u}_{L^{2\alpha'} (A(2\rho_j))}^2 
+ \norm{v}_{L^{2\alpha'} (A(2\rho_j))}^2 
+ \norm{\theta}_{L^{2\alpha'} (A(2\rho_j))}^2 \right) \\
&\lesssim \left( \norm{\nabla u}_{L^2(A(2\rho_j))} + \rho_j^{-\frac{1}{\alpha} + \frac{1}{6}} \right)^{\frac{6(3-\alpha)}{6-\alpha}}
+ \left( \norm{\nabla v}_{L^2(A(2\rho_j))} + \rho_j^{-\frac{3}{2\beta} + \frac{1}{4}} \right)^{\frac{6(2\alpha + \beta - \alpha\beta)}{\alpha(6-\beta)}} 
+ \left( \norm{\nabla \theta}_{L^2(A(2\rho_j))} + \rho_j^{-\frac{3}{2\gamma} + \frac{1}{4}} \right)^{\frac{6(2\alpha + \gamma - \alpha\gamma)}{\alpha(6-\gamma)}}.
\end{align*}
Employing \eqref{E310}, we have
\[
\rho_j^{-1} \norm{v}_{L^2(A(2\rho_j))} \norm{\theta}_{L^2 (A(2\rho_j))} 
\lesssim \left( \norm{\nabla v}_{L^2(A(2\rho_j))} + \rho_j^{-\frac{3}{2\beta} + \frac{1}{4}} \right)^{\frac{6-3\beta}{6-\beta}} 
\left( \norm{\nabla \theta}_{L^2(A(2\rho_j))} + \rho_j^{-\frac{3}{2\gamma} + \frac{1}{4}} \right)^{\frac{6-3\gamma}{6-\gamma}}.
\]
Now recall the assumption $3/2 < \alpha < 3$ and $1 \le \beta, \gamma < 2$.
Since all the exponents of $\rho_j$ are negative and $E(A(2\rho_j)) \to 0$ as $j \to \infty$, right side of \eqref{LEI2} goes to $0$ as $j \to \infty$, .
Therefore, we can conclude that $\nabla u = \nabla v = \nabla \theta = 0$ and hence $u, v, \theta$ are constant.
Since $G(u;\alpha,\rho_j) \le M$, we get $|u| \lesssim M \rho_j^{- \left( \frac{1}{\alpha} + \frac{1}{3} \right)}$, we should have $|u|=0$.
Similarly, we obtain that $v = \theta = 0$.
This completes the first part of Theorem \ref{T1}.

If we assume \eqref{A2}, then the proof is much simpler.
If $\alpha =3$ in \eqref{LEI} and \eqref{E31}, then we have 
\[
\lim_{j \to \infty} \left( G(u;3,\rho_j) + H(v;2,\rho_j) \right) = 0
\quad \text{and} \quad 
\sup_{j \in \mathbb{N}} H(\theta;2,\rho_j) \le M < \infty.
\]
Moreover, note that
\[
(R-r)^{-1} \norm{v}_{L^2(A(R))} \norm{\theta}_{L^2 (A(R))} \le (R-r)^{-1} \rho_j H(v;2,\rho_j) H(\theta;2,\rho_j).
\]
Repeating the same argument, we can easily obtain $u = v = \theta = 0$.
The proof for the case \eqref{A3} is similar to the case \eqref{A2}.
\qed

\section*{Acknowledgement}
This work was supported by the National Research Foundation of Korea(NRF) grant funded by the Korean government(MSIT) (No. 2021R1A2C4002840).

\bibliography{my}

\bibliographystyle{plain}

\end{document}